\documentclass[12pt, reqno]{amsart}

\usepackage[utf8]{inputenc}
\usepackage[T1]{fontenc}
\usepackage{tikz,tikz-3dplot}
\usepackage{amsmath,amsfonts,amssymb,amsmath,latexsym,mathrsfs}
\usepackage{shuffle}
\usepackage[all]{xy}
\usepackage{stackengine}
 
\usepackage{enumerate}

\usepackage{hyperref}
 
\usepackage{tabu}
\usepackage{tikz-cd} 

\usepackage{comment}
\usepackage{url}
\usepackage{tikz}

\usepackage{xfrac}

\newtheorem{theorem}{Theorem}[section]

\newtheorem{proposition}[theorem]{Proposition}

\newtheorem{corollary}[theorem]{Corollary}

\newtheorem{remark}[theorem]{Remark}

\newcommand{\N}{ {\mathbb N} }

\newcommand{\K}{ {\Bbbk} }
\newcommand{\e}{ {X} }
 \newcommand{\FL}{ {\mathbb L} }

\begin{document}

\title{one generator algebras over perfect fields}

\author{Mohamad \textsc{Maassarani} }
\maketitle
\begin{abstract}
For $R_1,R_2,R_3,\dots$ a family of non isomorphic rings (or algebras) having each only 2 idempotents ($1$ and $0$), we classify up to isomorphism the rings (or algebras) obtained by taking products of powers of the different $R_i$. We show that the automorphism groups of such rings (or algebras) split naturally into the product of wreath products $Aut( R_n)\wr \mathfrak{S}_{m_n} $ for different $n$. These results are applied to algebras generated by one element over a perfect field $\K$. Such algebra is either $\K[X]$ or a quotient of $\K[X]$. We show that in the later case the algebra is isomorphic to a finite product of the form $A=\prod (\FL_i[X]/(X^j))^{n_{i,j}}$, where the $\FL_i$ are non isomomorphic finite field extensions of $\K$ $($not isomophic as $\K$-algebras$)$, with restrictions on the numbers $n_{i,j}$ if $\K$ is finite. We classify these algebras up to isomorphism. We have also that the $\K$-algebra automorphism group of $A=\prod (\FL_i[X]/(X^j))^{n_{i,j}}$ splits naturally into the product of wreat products $Aut_\K(\FL_i[X]/(X^j) )\wr \mathfrak{S}_{n_{i,j}}$  ($Aut_\K(-)$ is for $\K$-algebra automorphism group). Finally, we prove that $Aut_\K(\FL_i[X]/(X^n) )$ is isomorphic to the semi-direct product $G_n(\FL_i)\rtimes Aut_\K(\FL_i)$  ($Aut_\K(-)$ is for $\K$-algebra automorphism group), where $G_n(\FL_i)\simeq Aut_{\FL_i}(\FL_i[X]/(X^n) )$ ($\FL_i$ algebra automorphism group) is an algebraic subgroup of invertible lower triangular matrices of dimension $(n-1)\times (n-1)$ with coefficients in $\FL_i$; the conjugate of a matrix $M\in G_n(\FL_i)$ by $\sigma \in Aut_\K(\FL_i)$ is the matrix obtained from $M$ by applying $\sigma$ to its coefficients.

\end{abstract}
\section*{Main results}
For $R_1,R_2,R_3,\dots$ a family of non isomorphic rings (or algebras) having each only 2 idempotents ($1$ and $0$), we show (section \ref{s1}) that  $A_{n_1,\dots,n_m}=R_1^{n_1}\times \cdots \times R_m^{n_m}$ (with $n_m\geq 1$)  is isomorphic to $A_{n_1',\dots,n_{m'}'}=R_1^{n_1'}\times \cdots \times R_{m'}^{n_{m'}'}$ (with $n_{m'}'\geq 1$) as rings (or algebras) then $m'=m$ and $n_i=n_i'$ for $i=1,\dots,m$.\\\\ In section \ref{s2}, we prove, under the same assumptions on $R_1,R_2,R_3,\dots$, that the group of ring (or algebra) automorphisms of $R_1^{n_1}\times \cdots \times R_m^{n_m}$ decomposes naturaly into the product of the wreath products $Aut( R_i)\wr \mathfrak{S}_{n_i}$ for $i\in \{1,\dots,m\}$ ($Aut(-)$ denotes the group of ring or algebra automorphisms). \\\\
Section \ref{s3} is devoted to $\K$-algebras generated by one element for $\K$ a perfect field. Such algebras are either isomorphic to $\K[X]$ either isomorphic to a quotient of $\K[X]$. We show that a quotient of $\K[X]$ is isomorphic to a finite propduct of the form $A=\prod (\FL_i[X]/(X^j))^{n_{i,j}}$, where the $\FL_i$ are non isomomorphic finite field extensions of $\K$ $($not isomophic as $\K$-algebras$)$, with restrictions on the numbers $n_{i,j}$ if $\K$ is finite. We classify these algebras up to isomorphism using results of section \ref{s1}. This is done in the first subsection of section \ref{s3}. In the second subsection, we consider the $\K$-algebra automorphism group $Aut_\K(A)$ of $A=\prod (\FL_i[X]/(X^j))^{n_{i,j}}$. Using results of section \ref{s2}, we get that $Aut_\K(A)$ splits naturally into the product of wreath products $Aut_\K(\FL_i[X]/(X^j) )\wr \mathfrak{S}_{n_{i,j}}$, where $Aut_\K(\FL_i[X]/(X^j))$ is the $\K$-algebra automorphism group of $\FL_i[X]/(X^j)$. We show that, for $\FL$ a finite field extension of $\K$, the group $Aut_\K(\FL[X]/(X^n))$ ($\K$-algebra automorphisms) is isomorphic to the semi-direct product $Aut_\FL(\FL[X]/(X^n)) \rtimes Aut_\K(\FL)$, where $Aut_\K(\FL)$ is the group of field automorphisms of $\FL$ fixing all the elements of $\K$ and $Aut_\FL(\FL[X]/(X^n))$ is the group of $\FL$-algebra automorphisms of $\FL[X]/(X^n)$. We also prove that $Aut_\FL(\FL[X]/(X^n))$ is isomorphic to an algebraic subgroup $G_n(\FL)$ of invertible lower triangular matrices of dimension $(n-1)\times (n-1)$ with coefficients in $\FL$. We hence have that $Aut_\K(\FL[X]/(X^n))\simeq G_n(\FL)\rtimes Aut_\K(\FL)$; the conjugate of a matrix $M\in G_n(\FL)$ by $\sigma \in Aut_\K(\FL)$ is the matrix obtained from $M$ by applying $\sigma$ to its coefficients.\\\\
We note that cyclic homology of one generator algebras is consider in the litterature (\cite{CC}, \cite{K}, \cite{CH}). 
 \tableofcontents

\section{Product of powers of non isomorphic rings or algebras with $2$ idempotents}\label{s1}
In this section $R_1,R_2,R_3,\dots$ is a family of non isomorphic rings (or algebras) having each only 2 idempotents ($1$ and $0$). We will study the rings (or algebras) $$A_{n_1,\dots,n_m}=R_1^{n_1}\times\cdots \times R_m^{n_m},$$ with $m\geq 1$ and $n_m \geq 1$.
 We will prove that If $A_{n_1,\dots,n_m}$ is isomorphic to $A_{n_1',\dots,n_{m'}'}$ then $m'=m$ and $n_i=n_i'$ for $i=1,\dots,m$. This is done by considering idempotents. Elements of this section are used in the next section.\\\\
 We fix an algebra $A_{n_1,\dots,n_m}=R_1^{n_1}\times \cdots \times R_m^{n_m},$ with $m\geq 1$ and $n_m \geq 1$.
 For $i\in \{1,\dots,m\}$ and $l\in\{1,\dots,n_i\}$, we denote by $1_{i,j}$ the unit of the $j$-th copy of $R_i $ in $A_{n_1,\dots,n_m}$ viewed as an element of $A_{n_1,\dots,n_m}$ ; for presicion $1_{i,j}$ is of the form $ (0,\dots,0,1,0,\dots,0)$. The elements $1_{i,j}$ we defined are idempotent elements of $A_{n_1,\dots,n_m}$ ($1_{i,j}^2=1_{i,j}$).

\begin{proposition}\label{decomp}
The idempotents of $A_{n_1,\dots,n_m}$ are the elements of the form : 
$$\overset{m}{\underset{i=1}{\sum}} \underset{j\in X_i}{\sum} 1_{i,j}, $$
where $X_i$ is a subset of $\{1,\dots,n_i\}$ for $i=1,\dots,m$.
\end{proposition}
\begin{proof}
Le $a=(a_1,\dots , a_{n_1+\dots+n_m})$ be an element of $A_{n_1,\dots,n_m}$. We have $a^2=(a_1^2,\dots , a_{n_1+\dots+n_m}^2)$. Hence, $a$ is idempotent if and only if each $a_i$ is idempotent in the corresponding $R_j$. But the only idempotents of $R_j$ are $1$ and $0$. This proves the proposition.
\end{proof}

\begin{corollary}\label{num idem}
The number of idempotents of $ A_{n_1,\dots,n_m}$ is $2^{n_1+\cdots+n_m}$.
\end{corollary}

For $a\in A_{n_1,\dots,n_m}$ we define the abelian group $E_1(a)$ by : $$E_1(a)=\{x \in A_{n_1,\dots,n_m}\mid ax=x \}.$$

\begin{proposition}\label{E1} \quad
\baselineskip 18 pt
\begin{itemize}
\item[1)] For $i \in \{1,\dots,m\}$ and $j\in \{1,\dots,n_i\}$: $$E_1(1_{i,j})=(R_i)_j,$$ where $(R_i)_j$ denotes the $j$-th copy of $R_i$ in the product of rings (or algebras) $A_{n_1,\dots,n_m}$.\\
\item[2)] For $i=1,\dots,m$ let $X_i$ be a subset of $\{1,\dots,n_i\}$. $$E_1(\overset{m}{\underset{i=1}{\sum}} \underset{j\in X_i}{\sum} 1_{i,j}) \simeq \overset{m}{\underset{i=1}{\prod}} R_i^{\vert X_i\vert},$$ where $\vert X_i\vert$ is the cardinal of $X_i$. 
\end{itemize}
\end{proposition}
\begin{proof}
This can be readly checked.
\end{proof}

\begin{remark}
For an idempotent $a\in A_{n_1,\dots,n_m}$ the abelian group $E_1(a)$ is a ring (or an algebra).
\end{remark}
 
\begin{proposition}\label{idem}
Let $a\in  A_{n_1,\dots,n_m}$ be an idempotent. Fix $i$ in $\{1,\dots,m\}$ and $j$ in $\{1,\dots,n_i\}$. There is a multiplicative and additive (linear if we have an agebra)  isomorphism between $E_1(a)$ and $E_1(1_{i,j})$ if and only if $a=1_{i,k}$ for a given $k\in\{1,\dots,n_i\}$. 
\end{proposition}
\begin{proof}
Let $a$ be and idempotent of $ A_{n_1,\dots,n_m}$ By proposition \ref{decomp} : 
\begin{equation}\label{1}
a=\overset{m}{\underset{i=1}{\sum}} \underset{j\in X_i}{\sum} 1_{i,j}, 
\end{equation}
where $X_i$ is a subset of $\{1,\dots,n_i\}$ for $i=1,\dots,m$. 
Applying proposition \ref{E1} and then corollary \ref{num idem} to $a$, we get that the number of idempotents in $E_1(a)$ is equal to $2^{\vert X_1\vert +\cdots \vert X_m \vert }$. By proposition \ref{E1} the ring $E_1(1_{i,j})$ is isomorphic to $R_i$. The ring (or algebra) $R_i$  has exactly two idempotents. In order to $E_1(1_{i,j})$ and $E_1(a)$ be isomorphic by an additive and multiplicative map, they should have the same number of idempotents and therefore $\vert X_i\vert +\cdots +\vert X_m \vert $ should be equal to $1$. This reduces the sum formula (\ref{1}) to $a=1_{r,k}$ for a given $(r,k)$ and we have $E_1(a)\simeq R_r$ by proposition \ref{E1}. Finally, $R_r$ and $R_i$ are isomorphic by an additive (linear if we have a algebra) multiplicative map (and therefore a ring or algebra isomorphism since $1$ is the only non trivial idempotent) if and only if $i=r$ and hence $a=1_{i,k}$ for a given $k$. This proves the proposition.
\end{proof}
\begin{proposition}\label{iso}
If $A=A_{n_1,\dots,n_m}$ is isomorphic to $A'=A_{n_1',\dots,n_{m'}'}$ as a ring (or algebra) then $m'=m$ and $n_i=n_i'$ for $i\in \{1,\dots,m\}$.
\end{proposition}
\begin{proof}
We denote by $1_{k,l}'$ the idempotent of $A'$ defined in similar fashion to the idempotent $1_{i,j}$ of $A$. Let $f:A\to A'$ be an isomorphism. The map $f$ gives an isomorphism between $E_1(1_{i,j})$ and $E_1(f(1_{i,j}))$. Since the image of an idempotent is an idempotent, one can argue as in the proof of proposition \ref{idem} to deduce that there should be an idempotent $1_{i,k_j}'\in A'$ image of $1_{i,j}$ by $f$. Moreover, since $f$ is injective, one should have at least $n_i$ elments of the form $1_{i,k_j}'$. This prooves that $m'\geq m$ and that $n_{i}'\geq n_i$ for $i=1,\dots,m$. Applying the same reasoning to $f^{-1}$ we get that $m\geq m'$ and that $n_{i}\geq n_i'$ for $i=1,\dots,m'$. Therefore, $m=m'$ and $n_i=n_i'$ for $i=1,\dots,m$.
\end{proof}

\section{Decomposition of the automorphism group}\label{s2}
As in the previous section $R_1,R_2,R_3,\dots$ is a family of non isomorphic rings (or algebras) having each only 2 idempotents ($1$ and $0$). We show, in this section, that the ring (or algbera) automorphism group of $$A_{n_1,\dots,n_m}=R_1^{n_1}\times \cdots \times R_m^{n_m},$$ with $m\geq 1$ and $n_m \geq 1$, decomposes naturallry  into the product of the wreath products $Aut( R_i)\wr \mathfrak{S}_{n_i} $ for $i\in \{1,\dots,m\}$, where $Aut(R_n)$ denotes the group of ring (or algebra) isomorphisms of $R_n$.  Elements of the previous section are used.\\\\
 We fix an algebra $A_{n_1,\dots,n_m}=R_1^{n_1}\times R_2^{n_2}\times \cdots \times R_m^{n_m},$ with $m\geq 1$ and $n_m \geq 1$. Throught this section, for $B$ a ring (or algebra), we denote by $Aut(B)$ the ring (or algebra) automorphism group of $B$.  

\begin{proposition}\label{per}
For any $ f\in Aut(A_{n_1,\dots,n_m})$ and $i\in\{1,\dots,m\}$ : 
\begin{itemize}
\item[1)] $f$ pemurtes isomorphically the $n_i$ factors isomorphic to $R_i$ in the product $A_{n_1,\dots,n_m}$.
\item[2)] $f$ stabilizes the factor $R_i^{n_i}$ and restricts to an automorphism $f_i :  R_i^{n_i}\to  R_i^{n_i} $.
\end{itemize}
\end{proposition}
\begin{proof}
Let $f$ be an automorphism of $A_{n_1,\dots,n_m}$. We recall that for $a\in  A_{n_1,\dots,n_m}$,  $E_1(a)=\{x \in A_{n_1,\dots,n_m} \mid ax=x \}$.  Let $1_{i,j}$ be the idempotent defined in the previous section. $f(1_{i,j})$ is idempotent (image of idempotent) and $f$ restricts to a multiplicative additive (or linear) isomorphism $g: E_1(1_{ij})\to E_1(f(1_{i,j}))$. Hence, by proposition \ref{idem}, $f(1_{i,j})=1_{i,k}$ for a given $k\in\{1,\dots,n_i\}$. Since $f$ is injective, it will permute the elements $1_{i,j}$ for $j\in\{1,\dots,n_i\}$  and therefore it premutes (isomorphically) the subspaces $E_1(1_{i,j})$ for $j\in\{1,\dots,n_i\}$. The first statement of the proposition follows because $E_1(1_{i,j})=(R_i)_j$ the $j$-th copy of $R_i$ in the product $A_{n_1,\dots,n_m}$ (see proposition \ref{E1}). The second statement of the proposition follows from the first. 
\end{proof}

The second statement of the above proposition allows to construct a natural morphism by restricting $f$ :
$$
  \begin{aligned}
    {\Psi} : Aut(A_{n_1,\dots,n_m}) & \longrightarrow Aut(R_1^{n_1})\times \cdots \times Aut(R_m^{n_m})  \\ 
    f & \longmapsto (f_1,\dots,f_m)
  \end{aligned}
 $$
where $f_i$ are as in the proposition the restrictions of $f$.

\begin{proposition}\label{Psi}
The morphism $\Psi$ above is an isomorphism with inverse : 
 $$
  \begin{aligned}
    {\Phi} :Aut(R_1^{n_1})\times \cdots \times Aut(R_m^{n_m}) & \longrightarrow Aut(A_{n_1,\dots,n_m})  \\
     (f_1,\dots,f_m) & \longmapsto f_1\times \cdots \times f_m
  \end{aligned}.
 $$

\end{proposition}
\begin{proof}
This can be readly checked.
\end{proof} 
 
We have decomposed naturally $Aut(A_{n_1,\dots,n_m}) $ into the product : 
$$ Aut(R_1^{n_1})\times \cdots \times Aut(R_m^{n_m}).$$
We will now decompose each $Aut(R_i^{n_i})$ for $i\in \{1,\dots,m\}$.\\\\
Let $n$ be a positive integer and $B$ a ring (or an algebra). For $\sigma \in \mathfrak{S}_n$  (symmetric group), we define the automorphism $f_\sigma \in Aut(B^n)$ by :
$$ f_\sigma(b_1,\dots,b_n)=(b_{\sigma^{-1}(1)}, \dots, b_{\sigma^{-1}(n)}),$$
for $b_1,\dots,b_n \in B$. We denote by $Aut(B) \wr\mathfrak{S}_n=(Aut(B))\rtimes \mathfrak{S}_n$ the semidirect product (wreath product) given by : 
$$ \sigma (f_1,\dots,f_n) =(f_{\sigma^{-1}(1)}, \dots, f_{\sigma^{-1}(n)})\sigma,$$
for $\sigma \in \mathfrak{S}_n$ and $f_1,\dots,f_n\in Aut(B)$.

\begin{proposition}
The map $\Theta : Aut(B) \wr \mathfrak{S}_n \to Aut(B^n)$ given by : 
$$ \Theta((f_1,\dots,f_n),\sigma))=(f_1\times \cdots \times f_n)\circ f_\sigma,$$
for $f_1,\dots , f_n\in Aut(B)$ and $\sigma \in \mathfrak{S}_n$ is an injective group morphism.
\end{proposition}
\begin{proof}
We first prove that $\Theta$ is a morphism. For $(f_1,\dots,f_n)$ and $(g_1,\dots, g_n) $ in $Aut(B)^n $, $\sigma$ and $\sigma'$ in $\mathfrak{S}_n$, we have : 
$$\begin{aligned}
\Theta(\ ((f_1,\dots,f_n), \sigma) ((g_1,\dots,g_n),\sigma ') \ ) =&\Theta(\ ((f_1\circ g_{\sigma^{-1}(1)},\dots,f_n\circ g_{\sigma^{-1}(n)}),\sigma\sigma ') \ )\\
=&(f_1\circ g_{\sigma^{-1}(1)}\times \cdots \times f_n\circ g_{\sigma^{-1}(n)}) \circ f_{\sigma\sigma '},
\end{aligned}$$
and 
$$\begin{aligned}
\Theta(\ ((f_1,\dots,f_n), \sigma) \ ) \circ \Theta(\ ((g_1,\dots,g_n),\sigma ') \ ) =&(f_1\times \cdots \times f_n)\circ f_\sigma \circ (g_1\times \cdots \times g_n)\circ f_{\sigma'}\\
=&(f_1\times \cdots \times f_n)\circ (f_\sigma \circ (g_1\times \cdots \times g_n) f_{\sigma^{-1}})\circ f_{\sigma\sigma'} \\
=&(f_1\times \cdots \times f_n)\circ ( g_{\sigma^{-1}(1)}\times \cdots \times g_{\sigma^{-1}(n)})\circ f_{\sigma\sigma'}\\
=&(f_1\circ g_{\sigma^{-1}(1)}\times \cdots \times f_n\circ g_{\sigma^{-1}(n)}) \circ f_{\sigma\sigma '}.
\end{aligned}$$
Comparing the results of both equations, we obtain that $\Theta$ is a morphism. We now prove that $\Theta $ is injective. Asume that $\Theta(\ ((f_1,\dots,f_n), \sigma) \ ) =\text{id}_{B^n}$. This means that for all $b_1,\dots,b_n \in B$, we have 
$$ (f_1(b_{\sigma^{-1}(1)}),\dots, f_n(b_{\sigma^{-1}(n)}))=(b_1,\dots,b_n).$$ 
If $\sigma ^{-1}(i) \neq i$, we get that for all $b\in B$ $a_i=f_i(b)$. This leads to a contradiction because $f_i$ is a bijection. Therefore, $\sigma=\text{id}_{\{1,\dots, n \}}$ and the above equation reduces to  
$$(f_1(b_1),\dots, f_n(b_n))=(b_1,\dots,b_n), $$
for all $b_1,\dots b_n \in B$. This implies that $f_1=\cdots= f_n= \text{id}_B$. We have shown that if $\Theta(\ ((f_1,\dots,f_n), \sigma) \ ) =\text{id}_{B^n}$, then $((f_1,\dots,f_n),\sigma)=((\text{id}_B,\dots, \text{id}_B),\text{id}_{\{1,\dots,n\}})$, proving that $\Theta$ is injective.
\end{proof}

\begin{proposition}
For $B=R_i$ the natural morphism $\Theta : Aut(R_i) \wr \mathfrak{S}_n \to Aut(R_i^n)$ of the previous proposition is an isomorphism.
\end{proposition}
\begin{proof}
Since we have shown in the previous proposition that $\Theta$ is injective we only need to prove that $\Theta$ is surjective. Let $f$ be an automorphism of $R_i^{n}$. The ring (or algebra) $R_i^n$ is in fact canonically isomorphic to $A_{0,\dots,0,n}$ where zero is repeated $i-1$ times. Hence, by proposition \ref{per} $f$ permutes the $n$ factors of $R_i$ in $R_i^n$. Consequently there exist $\sigma \in \mathfrak{S}_n$ and $f_1,\dots,f_n\in Aut(R_i)$, such that :
$$ f(b_1,\dots,b_n)=(f_1(b_{\sigma(1)}),\dots, f_n(b_{\sigma(n)})),$$
for all $(b_1,\dots, b_n) \in R_i^n$. Hence, 
$$ f=(f_1\times \cdots \times f_n)\circ f_{\sigma^{-1}}=\Theta(\ ((f_1\times \cdots \times f_n),\sigma^{-1}) \ ) .$$ We have proved that any $f\in Aut(R_i^n)$ lies in the image of $\Theta$ (i.e $\Theta$ is surjective). 
\end{proof}

\begin{proposition}\label{wr}
The automrphism group of $A_{n_1,\dots,n_m}=R_1^{n_1}\times \cdots \times R_m^{n_m}$ is naturally isomorphic to : 
$$ (Aut(R_1)\wr \mathfrak{S}_{n_1})\times  \cdots \times (Aut( R_m)\wr \mathfrak{S}_{n_m})$$
\end{proposition}
\begin{proof}
The result is obtained by combining proposition \ref{Psi} with the previous proposition.
\end{proof}

\section{Algebras generated by one element over a perfect field}\label{s3}
In this section $\K$ is a perfect field. This section is divided into $2$ subsections. In the first subsection, we classify $\K$-algebras generated by one element. In the second subsection, we study the automorphism group of such algebras. 
\subsection{Classification of algebras with one generator over a perfect field}We will classify the $\K$-algebras generated by one element over the prefect field $\K$.\\\\
The following result is theorem 1.7 of \cite{MM} under the assumption $\K$ is perfect.
\begin{theorem}[\cite{MM}]\label{e}
Let $P$ be an irreducible polynomial over $\K$ $($perfect$)$. For $n\geq 0$, the $\K$-algebra $\K[X]/(P^n)$ is isomorphic to $(\K[X]/(P))[Y]/(Y^n)$.
\end{theorem}
 
Let $\FL$ be a finite field extension of $\K$. We denote by $\varphi(\FL)$ the number of monic irreducible polynomials $P$ such that $\K[X]/(P)$ is isomorphic to $\FL$ as a $\K$-algebra. 
 
\begin{proposition}
For $\FL$ a finite extension of $\K$ $($perfect$)$. The number $\varphi(\FL)$ is greater or equal to $1$. $\varphi(\FL)$ is finite if and only if $\K$ is finite.
\end{proposition}
\begin{proof}
Since $\K$ is perfect by the primitive element theorem there is a polynomial $P\in \K[X]$ such that $\FL \simeq \K[X]/(P)$ as a $\K$-algebra. This proves that $\varphi(\FL) \geq 1$. If $\K$ is finite there are finitely many irreducible polynomials of degree equal to the degree of the extension $\FL/\K$, hence $\varphi(\FL)$ is finite. The polynomials $P(X)$ and $P(X+a)$ for $a\in \K$ define isomorphic filed extensions. Hence, if $\K$ is infinite then $\varphi(\FL)$ is infinite.   
\end{proof}
We recall that there is a formula for $\varphi(\FL)$ when $\FL$ is finite. 
For $m\geq 1$, $\underline{n}_m$ will designate an $m$-tuple of positive integers $(n_1,\dots,n_m)$ such that $n_m\geq 1$ and we set $\vert \underline{n}_m \vert=n_1+\cdots+n_m$. 
For $\FL$ a finite extension of $\K$ and an $m$-tuple $\underline{n}_m$ we define the $\K$-algebra (and $\FL$-algebra) $O_{\underline{n}_m}(\FL)$ by :
 $$O_{\underline{n}_m}(\FL)=\FL^{n_1}\times (\FL[X]/(X^2))^{n_2}\times \cdots \times (\FL[X]/(X^m))^{n_m}.$$ 

\begin{theorem}\label{A}
Let $A$ be a $\K$-algebra generated by one element. $A$ is either isomorphic to $\K[X]$ either isomorphic to :
$$A_{\underline{n}_{m_1},\dots, \underline{n}_{m_k}}(\FL_1,\dots ,\FL_k)=O_{\underline{n}_{m_1}}(\FL_1)\times \cdots \times O_{\underline{n}_{m_k}}(\FL_k),$$
Where $k\geq 1$, $\FL_1,\dots,\FL_k$ are non isomorphic finite field extensions of $\K$ $($not isomophic as $\K$-algebras$)$ and $\vert \underline{n}_{m_i} \vert \leq \varphi(\FL_i)$, for $i=1,\dots,k$. Moreover, any such $A_{\underline{n}_{m_1},\dots, \underline{n}_{m_k}}(\FL_1,\dots ,\FL_k)$ is generated as a $\K$-algebra by one element.
\end{theorem}
\begin{proof}
Since $A$ is generated by one element, $A$ is either isomorphic to $\K[X]$ either isomorphic to a quotient $\K[X]/(P)$ for some $P\in \K[X]$. In the latter case, denote by $P_1,\dots,P_r$ the monic irreducible factors of $P$. Let $\FL_1,\dots,\FL_k$ be representatives of the $\K$-algebra isomorphism classes of the fields $\K[X]/(P_1),\dots,\K[X]/(P_r)$ and denote by $X(\FL_i)$ the set of irreducible factors giving a field extension isomorphic (as a $\K$-algebra) to $\FL_i$. $\FL_1,\dots,\FL_k$ are not isomorphic as $\K$-algebras, $X(\FL_1),\dots,X(\FL_k)$ are disjoint and the cardinal of $X(\FL_i)$ is less then $\varphi(\FL_i)$, for $i=1,\dots,k$. We can write :
$$ P=\overset{k}{\underset{i=1}{\prod}} \underset{Q\in X(\FL_i)}{\prod} Q^{n(Q)},$$
for some $n(Q)$. Applying the chinese reminder theorem to $\K[X]/(P)$ then theorem \ref{e} we get that :
$$ \K[X]/(P) \simeq \overset{k}{\underset{i=1}{\prod}} \underset{Q\in X(\FL_i)}{\prod} \K[X]/ (Q^{n(Q)})\simeq \overset{k}{\underset{i=1}{\prod}} \underset{Q\in X(\FL_i)}{\prod} \FL_i[X]/ (X^{n(Q)}). $$
As we have seen $\vert X(\FL_i) \vert \leq \varphi(\FL_i)$, hence the last equation proves that $A$ is isomorphic to a $\K$-algebra of the form $A_{\underline{n}_{m_1},\dots, \underline{n}_{m_k}}(\FL_1,\dots ,\FL_k)$ as in the proposition. We now prove the "moreover" part of the proposition. We fix $A_{\underline{n}_{m_1},\dots, \underline{n}_{m_k}}(\FL_1,\dots ,\FL_k)$ as in the proposition. Since $\vert n_{m_i} \vert \leq \varphi(\FL_i)$, one has distinct irreducible polynomials $Q_{i,1},\dots,Q_{i,\vert \underline{n}_{m_i} \vert}$ such that $\K[X]/Q_{i,j} \simeq \FL_i$ for $i=1,\dots,k$. For $0\leq l< m_i$, Set $$Q_i^{(l)}=Q_{i,n_{1,i}+\dots+n_{l,i}+1}\cdots Q_{i,n_{1,i}+\dots+n_{l+1,i}},$$where $(n_{1,i},\dots, n_{m_i,i})=\underline{n}_{m_i}$ and set $$ Q_i=\overset{m_i-1}{\underset{l=0}{\prod}}(Q_i^{(l)})^{l+1}\quad \text{and} \quad Q=\overset{k}{\underset{i=1}{\prod}}Q_i.$$
One has by the chinese reminder theorem and theorem \ref{e} : \begin{align*} \K[X]/(Q)\simeq\overset{k}{\underset{i=1}{\prod}}\K[X]/(Q_i)&\simeq \overset{k}{\underset{i=1}{\prod}}\overset{m_i-1}{\underset{l=0}{\prod}}\K[X]/((Q_i^{(l)})^{l+1})\\&\simeq \overset{k}{\underset{i=1}{\prod}}\overset{m_i-1}{\underset{l=0}{\prod}}(\FL_i[X]/(X^{l+1}))^{n_{l+1,i}}\\& \simeq  \overset{k}{\underset{i=1}{\prod}} O_{\underline{n}_{m_i}}(\FL_i) .\end{align*}
This proves that $\K[X]/(Q)$ is isomorphic to $A_{\underline{n}_{m_1},\dots, \underline{n}_{m_k}}(\FL_1,\dots ,\FL_k)$ and hence the later algebra is generated by one element.
\end{proof}

\begin{proposition}\label{id}

\begin{itemize}
\item[1)] Let $\FL$ be a field. The idempotents of $\FL[X]/(X^n)$ are the unit and zero.
\item[2)] Let $\FL$ and $\FL'$ be finite field extensions of $\K$. $\FL[X]/(X^n)$ is isomorphic to $\FL'[X]/(X^m)$ as $\K$-algebras if and only if $\FL\simeq \FL'$ as $\K$-algebras and $n=m$
\end{itemize}
\end{proposition}
\begin{proof}
 We prove 1). $(1,X,X^2 , \dots,\e^{n-1}) $ is a basis of $\K[X]/(X^n)$. Therefore, any element of $\K[X]/(X^n)$ is of the form $\lambda_0 + \lambda_1 \e + \cdots + \lambda_{n-1} \e^{n-1}$ with $\lambda_i \in \K$ ($i=0,\dots,n-1$). Such an element is idempotent if and only if $$(\lambda_0 + \lambda_1 \e + \cdots + \lambda_{n-1} \e^{n-1})^2=\lambda_0 + \lambda_1 \e + \cdots + \lambda_{n-1} \e^{n-1}.$$  By developping the square we get the equation : $$\alpha_0 + \alpha_1 \e + \cdots + \alpha_{n-1} \e^{n-1}=  \lambda_0 + \lambda_1 \e + \cdots + \lambda_{n-1} \e^{n-1} \quad \text{where}\quad \alpha_i=\underset{k+j=i}{\sum} \lambda_k \lambda_j,$$ and hence we have the set of equations : $$ \lambda_i  =  \underset{k+j=i}{\sum} \lambda_k \lambda_j \quad \text{for} \quad i=0,\dots,n-1.$$ For $i=0$, we get the equation $\lambda_0=\lambda_0^2$ with two solutions $\lambda_0=1$ or $\lambda_0=0$. For $i=1$, we find that $\lambda_1=0$ for both values of $\lambda_0$ we found. Using an induction on the equations one shows that $\lambda_i=0$ for $i\geq 2  $ and proves the first point of the proposition. We now prove $2)$. $\FL[X]/(X^n)$ and $\FL'[X]/(X^m)$ are local rings with respective residue fields $\FL$ and $\FL'$. Hence, if $\FL[X]/(X^n)$ and $\FL'[X]/(X^m)$ are isomorphic as $\K$-algebras then $\FL\simeq \FL'$ as $\K$-algebras and $\FL$ and $\FL'$ has the same dimension $d$ over $\K$. Finally, the dimensions of $\FL[X]/(X^n)$ and $\FL'[X]/(X^m)$ over $\K$ are respectivly $dn$ and $dm$ and they should be equal if we have a $\K$-algebra isomorphism $\FL[X]/(X^n)\simeq \FL'[X]/(X^m)$, hence, $n=m$. We proved the "if" part. The converse is clear. We have proved 2). 
\end{proof} 
Let $(\FL_i)_{i\in I}$ be a family of non isomorphic finite field extensions of $\K$ $($not isomophic as $\K$-algebras$)$. The proposition shows that the family of  $\K$-algebras   $(\FL_i[X]/(X^n))_{i\in I,n \in \N}$ is a family of non isomorphic $\K$-algebras having each exactly two idempotents ($1$ and $0$). Thus, the results of section \ref{s1} apply. We get from proposition \ref{iso} the following :
\begin{theorem}\label{B}
Let $A=A_{\underline{n}_{m_1},\dots, \underline{n}_{m_k}}(\FL_1,\dots ,\FL_k)$ and $A'=A_{\underline{n}_{m'_1},\dots, \underline{n}_{m'_{k'}}}(\FL'_1,\dots ,\FL'_{k'})$ be $\K$-algebras as in theorem \ref{A}. $A$ and $A'$ are isomorphic $\K$-algebras if and only if $k=k'$ and there exist a permutation $\sigma : \{1,\dots ,k \} \to \{1,\dots,k \}$ such that $\FL'_i\simeq \FL_\sigma(i)$ $($as $\K$-algebras$)$ and $\underline{n}_{m'_i}=\underline{n}_{m_{\sigma(i)}}$.
\end{theorem}

\subsection{Automorphism group of algebras with one generator over a perfect field} 
We will consider the $\K$-algebra automorphism group of the algebras $A_{\underline{n}_{m_1},\dots, \underline{n}_{m_k}}(\FL_1,\dots ,\FL_k)$ described in theorem \ref{A}.  \\\\
As we have seen (previous subsection) the family of $\K$-algebras $(\FL_i[X]/(X^n))_{i\in I,n \in \N}$, where $(\FL_i)_{i\in I}$ is a family of non isomorphic finite field extensions of $\K$ $($not isomophic as $\K$-algebras$)$, is a family of non isomorphic $\K$-algebras having each exactly two idempotents ($1$ and $0$). Thus, the results of sections \ref{s2} apply. So an automorphism of $A_{\underline{n}_{m_1},\dots, \underline{n}_{m_k}}(\FL_1,\dots ,\FL_k)$ permutes the isomorphic factors (see proposition \ref{per} ) and we have by proposition \ref{wr} the natural decomposition :

\begin{theorem} 
The group of $\K$ algebra automorphisms of $A_{\underline{n}_{m_1},\dots, \underline{n}_{m_k}}(\FL_1,\dots ,\FL_k)$ is naturally isomorphic to :$$\overset{k}{\underset{i=1}{\prod}} \overset{m_i}{\underset{j=1}{\prod}}Aut_\K(\FL_i[X]/(X^j))\wr \mathfrak{S}_{n_{j,i}},$$
where $Aut_\K(-)$ is for the $\K$-algebra automorphism group and $(n_{1,i},\dots,n_{m_i,i})=\underline{n}_{m_i}$.
\end{theorem}
We will now study the $\K$-algebra automorphism group $Aut_\K(\FL[X]/(X^n))$ of $\FL[X]/(X^n)$ for $\FL$ a finite field extension of $\K$. We first consider the $\FL$-algebra automorphism group $Aut_\FL(\FL[X]/(X^n))$ of $\FL[X]/(X^n)$.
\subsubsection{The group $Aut_\FL(\FL[X]/(X^n))$}In the following $\FL$ can be taken to be any field.
\begin{proposition}\label{O}
\begin{itemize}
\item[1)] If $f \in Aut_\FL(\FL[X]/(X^n))$ then $f(X)=\lambda_1 X + \cdots + \lambda_{n-1} X^{n-1},$ for a given $(\lambda_1,\dots,\lambda_{n-1}) \in \FL^*\times \FL^{n-2}$, where $\FL^*$ denotes the invertibles of $\FL$.
\item[2)] For $(\lambda_1,\dots,\lambda_{n-1}) \in \FL^*\times \FL^{n-2}$ there exist a unique automorphism  $f \in Aut_\FL(\FL[X]/(X^n))$ such that $f(X)$ is as in 1) of this proposition. 
\end{itemize}
\end{proposition}
\begin{proof}
Let $f$ be a $\FL$-algebra automorphism of $\FL[X]/(X^n)$.  $f(X)$ decomposes in the basis $(1,X,\dots,X^{n-1})$ as a sum :
$$f(X)=\lambda_0 +\lambda_1 X + \cdots + \lambda_{n-1} X^{n-1},$$ for $(\lambda_0, \lambda_1,\dots,\lambda_{n-1}) \in  \FL^n$. One checks that the condition $f(X^n)=f(X)^n=0$ is equivalent to $\lambda_0=0$. The condition $\lambda_1=0$ implies that $f(X^k)=0$ for a $k< n$ but $f$ is an automorphism and hence $\lambda_1\neq 0$. We have prooved 1). Now let $v=\alpha_1 X+ \cdots \alpha_{n-1} X^{n-1}$  with $(\alpha_1,\dots, \alpha_{n-1})\in\FL^* \times \FL^{n-2}$. Since $v$ has no constant term ("$\alpha_0=0$"), $v^n=0$ and hence there is a unique $\FL$-algebra morphism $f_v :  \FL[X]/(X^n) \to \FL[X]/(X^n)$ such that $f_v(X)=v$. The condition $\alpha_1 \in \FL^*$ insures that $f$ is invertible. Indeed $f_v(1)=1$ and $f_v(v^k)=\lambda_1^kX+\text{terms of higher degree in $X$}$, and hence the determinant of the matrix of $f_v$ in the basis $(1,X,\dots,X^{n-1})$ is a power of the invertible $\alpha_1$. We have prooved 2).
\end{proof}

For $(\lambda_1,\dots,\lambda_{n-1}) \in \FL^*\times \FL^{n-2}$, and $i,j\in\{1,\dots,n-1\}$ we set :
\begin{equation}\label{alpha}
\alpha_{ij}(\lambda_1,\dots,\lambda_{n-1})= \underset{i_1+\cdots+ i_j=i \vert i_l >0}{\sum} \lambda_{i_1}\cdots \lambda_{i_j}
\end{equation}

and we define the $(n-1) \times (n-1)$ matrix with coefficient in $\FL$ : $$A(\lambda_1,\dots,\lambda_{n-1})=(\alpha_{ij}(\lambda_1,\dots,\lambda_{n-1}))_{i,j}.$$ The matrix $A(\lambda_1,\dots,\lambda_{n-1})$ is lower triangular and looks like 
 \[\begin{bmatrix}
    \lambda_1 & 0   & \dots& \dots &  0 \\
      \lambda_2 & \lambda_1^2   &  0 &\dots  &0 \\
    \vdots & \vdots &  \ddots & \ddots &\vdots \\
\lambda_{n-2} &  \dot & \dots  &\ddots &0\\
    \lambda_{n-1} &   \dot & \dots  & &\lambda_1^{n-1}
\end{bmatrix}
\]
with coefficients that are polynomial in the coefficients $\lambda_1,\dots,\lambda_{n-1}$ of the first column. \\
Let $G_n(\FL)$ be the subset of lower triangular invertible matrix of dimension $(n-1)\times (n-1)$ with coefficient in $\FL$ defined by :
$$G_n(\FL)=\{A(\lambda_1,\dots,\lambda_{n-1}) \mid (\lambda_1,\dots,\lambda_{n-1}) \in \FL^*\times\FL^{n-2}\}.$$
\begin{proposition}\label{Gnp}
Let $\mathbf{B}$ be the basis $(1, X, \dots, X^{n-1})$ and $f$ be an element of $Aut_\FL(\FL[X]/(X^n))$ such that $$f(X)=\lambda_1 X + \cdots + \lambda_{n-1} X^{n-1},$$ for some $(\lambda_1,\dots,\lambda_{n-1}) \in \FL^*\times \FL^{n-2}$. The Matrix of $f$ in $\mathbf{B}$ is the matrix :
 \[\begin{bmatrix}
   1 &  0   & \cdots& &  0 \\
      0 & \  &    &   & \\
    \vdots &  &   A(\lambda_1,\dots,\lambda_{n-1})    &  \\
0 &  &   &   & 
\end{bmatrix}
\]

\end{proposition}
 \begin{proof}
The $j$-th column of the matrix of $f$ is determined by $f(X^j)$. For $j=0$, $f(1)=1$ and hence the first column is $$\begin{bmatrix}
   1  \\
      0  \\
    \vdots   \\
0 
\end{bmatrix}.$$ Now for $j>0$, $$f(X^j)=(\lambda_1 X + \cdots + \lambda_{n-1} X^{n-1})^j=\overset{n-1}{\underset{i=1}{\sum}} \alpha_{ij}(\lambda_1,\dots,\lambda_{n-1}) X^i, $$where $\alpha_{ij}(\lambda_1,\dots,\lambda_{n-1}) $ is defined in equation $(\ref{alpha})$. Hence, the $j$-th column is $$\begin{bmatrix}
   0  \\
      \alpha_{1,j}(\lambda_1,\dots,\lambda_{n-1})  \\
    \vdots   \\
 \alpha_{n-1,j}(\lambda_1,\dots,\lambda_{n-1}) 
\end{bmatrix}.$$
This proves the proposition.
\end{proof}
Combing the last proposition with the proposition before we get the corollary.

\begin{corollary}\label{Gn}
$G_n(\FL)$ is a group under matrix multiplication and its naturally isomorphic to $Aut_\FL(\FL[X]/(X^n))$. The isomorphism is griven by $f\mapsto A(\lambda_1,\dots,\lambda_{n-1})$ where $f(X)=\lambda_1X+\cdots+ \lambda_{n-1} X^{n-1}$.
\end{corollary}

There is a bijection  $\FL^*\times \FL^{n-2}\to G_n(\FL), (\lambda_1,\dots,\lambda_{n-1})\mapsto A(\lambda_1,\dots,\lambda_{n-1})$. Hence, the product over $G_n(\FL)$ can be transfered to give a group product $\star  $ over $\FL^*\times \FL^{n-2}$. The product $\star$ is given by : 

$$ (\lambda_1,\dots,\lambda_{n-1}) \star (\lambda_1',\dots,\lambda_{n-1}')=(\beta_1\dots,\beta_n) \quad \text{where} \quad \beta_k=\overset{k}{\underset{i=1}{\sum}} \alpha_{k,i}(\lambda_1,\dots,\lambda_{n-1}) \lambda_i',$$
and the $\alpha_{i,j}(\lambda_1,\dots,\lambda_{n-1})$ are the polynomials in $\lambda_1,\dots,\lambda_{n-1}$ defined by equation (\ref{alpha}).

\subsubsection{The group $Aut_\K(\FL[X]/(X^n))$}
In the following $\FL$ is a finite field extension of $\K$ (perfect). We will consider the $\K$-algebra automorphism group $Aut_\K(\FL[X]/(X^n))$ of $\FL[X]/(X^n)$. The group of $\K$-algebra automorphisms of $\FL$, (field automorphisms of $\FL$ fixing the elements of $\K$) will be denoted $Aut_\K(L)$
\begin{proposition}\label{sigma}
If $f\in Aut_\K(\FL[X]/(X^n))$ then there exist $\sigma_f \in Aut_\K(\FL)$ such that $f(a)=\sigma_f(a)$ for all $a\in \FL$. 
\end{proposition}
\begin{proof}
Take $a\in \FL$ and let $P_a$ be its minimal polynomial over $\K$. Since $f$ is $\K$-algebra morphism $P_a(f(a))=0$. Now assume that $f(a)=b+cX^k+o(X^k)$, with $b,c \in \FL$, $c\neq 0$, $k<n$ and where we use an will use (in this proof) the notation $o(X^k)$ for a sum (eventually $0$) of terms with degree strictly greater then $k$ in $X$ with coefficients in $\FL$. We have  :
$$P_a(f(a))=P_a(b+cX^k+o(X^k))=P_a(b)+cP_a'(b)X^k+o(X^k)=0.$$
This implies that $P(b)=P'(b)=0$, wich is not possible since $\K$ is perfect. Hence, $f(a)$ lies necessarly in $\FL$ and we have $f(\FL)\subset \FL$. Considering $f^{-1}$ we find that $f(\FL)=\FL$ and the proposition follows. 
\end{proof}
\begin{proposition}
\begin{itemize}
\item[1)] For $\sigma \in Aut_\K(\FL)$ there exist a unique $\K$-algebra automorphism $\sigma^X$ of $\FL[X]/(X^n)$ given by $\sigma^X(a)=\sigma(a)$ for $a\in \FL$ and $\sigma^X(X)=X$.
\item[2)] The map $\Phi :Aut_\K(\FL)\to Aut_\K(\FL[X]/(X^n)), \sigma \mapsto \sigma^X$ is an injective group morphism. 
\end{itemize}
\end{proposition}
\begin{proof}
One has a unique $\K$-algebra isomorphism $\tilde{\sigma} :\FL[X]\to \FL[X]$ given by $\tilde{\sigma}(a)=\sigma(a)$ for $a\in \FL$ and $\tilde{\sigma}(X)=X$. The isomorphism $\tilde{\sigma}$ maps the ideal $(X^n)$ onto itself and hence induces a $\K$-algebra automorphism $\sigma^X$ of $\FL[X]/(X^n)$. The uniqueness of $\sigma^X$, follow from the fact that $\FL[X]/(X^n)$ is generated by $\FL$ and $X$. 2) can be readly checked.
\end{proof}
\begin{proposition}
The group $Aut_\K(\FL[X](/X^n))$ is naturally isomorphic to the semi-direct product $Aut_\FL(\FL[X]/(X^n) )\rtimes_\theta Aut_\K(\FL)$, where $\theta(\sigma)(f)=\sigma^Xf(\sigma^X)^{-1}$ for $\sigma  \in Aut_\K(\FL)$ and $f \in Aut_\FL(\FL[X]/(X^n))$ and where $\sigma^X$ is as in the previos proposition.
\end{proposition}
\begin{proof}
$Aut_\FL(\FL[X]/(X^n))$ is a subgroup of $Aut_\K(\FL[X]/(X^n))$. By the previous proposition, the statement of the proposition is equivalent to the statement $Aut_\K(\FL[X]/X^n)$ splits as a semi-direct product of $Aut_\FL(\FL[X]/(X^n))$ and $\Phi(Aut_\K(\FL))$ ($\Phi$ of the previous proposition). We prove the later statement. First note that $Aut_\FL(\FL[X]/(X^n))\cap  \Phi(Aut_\K(\FL))=\{1\}$. By proposition \ref{sigma}, if $f\in Aut_\K(\FL[X]/(X^n))$ then there exist $\sigma_f \in Aut_\K(\FL)$ such that $f(a)=\sigma_f(a)$ for all $a\in \FL$. Notice that $f\circ (\sigma_f^X)^{-1}$ lies in $Aut_\FL(\FL[X]/(X^n))$. Indeed, $f\circ (\sigma_f^X)^{-1}(a)=\sigma_f(\sigma_f^{-1}(a))=a$ for $a\in \FL$. This proves that $Aut_\K(\FL[X]/(X^n))=Aut_\FL(\FL[X]/(X^n))\Phi (Aut_\K(\FL))$. Finally, for $\sigma  \in Aut_\K(\FL)$ and $f \in Aut_\K(\FL)$, $\sigma^Xf(\sigma^X)^{-1}$ restricts to the identity on $\FL$ and hence lies in $Aut_\FL(\FL[X]/(X^n))$. we have proved the proposition.
\end{proof}
We recall that by proposition \ref{O}, if $f\in Aut_\FL(\FL[X]/(X^n))$ then $f(X)= \lambda_1 X + \cdots + \lambda_{n-1} X^{n-1}$ for a given $(\lambda_1,\dots,\lambda_{n-1}) \in \FL^*\times \FL^{n-2}$.

\begin{proposition}
Take $f\in Aut_\FL(\FL[X]/(X^n))$. If $f(X)= \lambda_1 X + \cdots + \lambda_{n-1} X^{n-1}$ with $(\lambda_1,\dots,\lambda_{n-1}) \in \FL^*\times \FL^{n-2}$ then for $\sigma \in Aut_\K(\FL)$ and $\theta$ as in the previous proposition : 

$$ \theta(\sigma)(f)(X)=\sigma(\lambda_1)X+\cdots+\sigma(\lambda_{n-1}) X^{n-1}.$$
 
\end{proposition}
\begin{proof}
The equation follows from the fact that \begin{align*}\theta(\sigma)(f)(X)&=\sigma^Xf(\sigma^X)^{-1}(X)=\sigma^Xf(X)\\&=\sigma^X(\lambda_1 X + \cdots + \lambda_{n-1} X^{n-1})\\&=\sigma(\lambda_1)X+\cdots+\sigma(\lambda_{n-1}) X^{n-1}.\end{align*}
\end{proof}
We recall that $$G_n(\FL)=\{A(\lambda_1,\dots,\lambda_{n-1}) \mid (\lambda_1,\dots,\lambda_{n-1}) \in \FL^*\times\FL^{n-2}\},$$
is a group under matrix multiplication (corollary \ref{Gn}).
\begin{theorem}
We have a group isomorphism between $Aut_\K(\FL[X]/(X^n))$ and the semi-direct product $G_n(\FL)\rtimes_\beta Aut_\K(\FL)$, where $\beta(\sigma)(A(\lambda_1,\dots,\lambda_{n-1}))=A(\sigma(\lambda_1),\dots,\sigma(\lambda_{n-1}))$; $A(\sigma(\lambda_1),\dots,\sigma(\lambda_{n-1}))$ is the matrix obtained form $A(\lambda_1,\dots,\lambda_{n-1})$ by applying $\sigma$ to the coefficients.
\end{theorem}
\begin{proof}
By corollary \ref{Gn}, we have an isomorphism $ Aut_\FL(\FL[X]/(X^n)) \to G_n$ given by  $f\mapsto A(\lambda_1,\dots,\lambda_{n-1})$ where $f(X)=\lambda_1X+\cdots+ \lambda_{n-1} X^{n-1}$. Hence, by the last two propositions we get the first statement of the theorem. We recal that the coefficients of $A(\lambda_1,\dots,\lambda_{n-1})$ are the 
$$\alpha_{ij}(\lambda_1,\dots,\lambda_{n-1})= \underset{i_1+\cdots+ i_j=i \vert i_l >0}{\sum} \lambda_{i_1}\cdots \lambda_{i_j} .$$
Hence, the coefficients $\alpha_{ij}(\sigma(\lambda_1),\dots,\sigma(\lambda_{n-1}))$ of $A(\sigma(\lambda_1),\dots,\sigma(\lambda_{n-1}))$ are equal to $$\sigma(\alpha_{ij}(\lambda_1,\dots,\lambda_{n-1})).$$ This completes the proof of the theorem.
\end{proof}

 \end{document}